\documentclass{article}
\usepackage[T1]{fontenc}
\usepackage[latin9]{inputenc}
\usepackage{verbatim}
\usepackage{amsmath}
\usepackage{amsthm}
\usepackage{amssymb}

\makeatletter
\newcommand{\lyxaddress}[1]{
	\par {\raggedright #1
	\vspace{1.4em}
	\noindent\par}
}
\theoremstyle{plain}
\newtheorem{thm}{\protect\theoremname}
\theoremstyle{plain}
\newtheorem{prop}[thm]{\protect\propositionname}

\@ifundefined{date}{}{\date{}}
\makeatother

\providecommand{\propositionname}{Proposition}
\providecommand{\theoremname}{Theorem}

\begin{document}
\title{Some trigonometric integrals and the Fourier transform of a spherically
symmetric exponential function}
\author{Hideshi YAMANE}
\maketitle

\lyxaddress{Department of Mathematical Sciences, Kwansei Gakuin University \\
Gakuen 2-1, Sanda, Hyogo, Japan 669-1337\\
\texttt{yamane@kwansei.ac.jp}
}
\begin{abstract}
We  calculate the Fourier transform of a spherically symmetric exponential
function. Our evaluation is much simpler than the known one. We use
the polar coordinates and reduce the Fourier transform to the integral
of a rational function of trigonometric functions. Its evaluation
turns out to be much easier than expected because of homogeneity and
a hidden symmetry. Relationship with a Fourier integral representation
formula for harmonic functions is explained.
\end{abstract}

\section{Fourier transform }

In this article we give a simple proof of the following theorem about
the Fourier transform of a spherically symmetric function. 
\begin{thm}
\label{thm:Stein-Weiss}If $a>0$, we have 
\begin{equation}
\int_{\mathbb{R}^{n}}e^{-2\pi a|x|}e^{-2\pi it\cdot x}\,dx=\frac{a\Gamma[(n+1)/2]}{\pi^{(n+1)/2}(a^{2}+|t|^{2})^{(n+1)/2}}.\label{eq: Stein-Weiss}
\end{equation}
\end{thm}

The Fourier transform of $e^{-2\pi a|x|}$ is much harder to obtain
than that of the Gaussian kernel $e^{-\pi\alpha|x|^{2}}$. Lower dimensional
($n\le3$) cases can be found in many books. For example, the one-
and three-dimensional cases are \cite[12.23.9]{GR} and \cite[12.24]{GR}
respectively. The general case can be found in the standard textbook
\cite{Stein-Weiss}. The proof in \cite[pp.6-7]{Stein-Weiss} based
on the unfamiliar formula 
\begin{equation}
e^{-\beta}=\frac{1}{\sqrt{\pi}}\int_{0}^{\infty}\frac{e^{-u}}{\sqrt{u}}e^{-\beta^{2}/4u}\,du,\quad\beta>0\label{eq:strangeformula}
\end{equation}
is technical and mysterious. We give a much simpler proof in this
article. Since we are dealing with a spherically symmetric function
$e^{-2\pi a|x|}$, it is natural to use the polar coordinates. Then
we are led to integrals involving trigonometric functions. Their evaluation
is interesting in its own right.

First we recall a result about the integral of the $n$-th power of
$\sin x$ (\cite[3.621]{GR}).
\begin{prop}
\label{prop:highschool} Set $\gamma_{n}=\int_{0}^{\pi}\sin^{n}\theta\,d\theta$
for $n\in\mathbb{N}=\{0,1,2,\dots\}$. Then we have $\gamma_{n}=2(n-1)!!/n!!$
if $n$ is odd and $\gamma_{n}=\pi(n-1)!!/n!!$ if $n$ is even. (By
convention, $(-1)!!=0!!=1$.) We also have 
\[
\gamma_{n}=2^{n}B\left(\frac{n+1}{2},\frac{n+1}{2}\right).
\]
\end{prop}

The following proposition will be proved in Section \ref{sec:Integrating-a-rational}.
\begin{prop}
\label{prop: integtrig} For $a>0,p\in\mathbb{R},$ we have 
\begin{equation}
\int_{0}^{\pi}\frac{\sin^{n}x}{(a+ip\cos x)^{n+1}}\,dx=\frac{\gamma_{n}}{(a^{2}+p^{2})^{(n+1)/2}}\;(n=0,1,2,\dots).\label{eq:integtrig}
\end{equation}
\end{prop}

Now we prove Theorem \ref{thm:Stein-Weiss}. Set $x=r\xi,\xi\in S^{n-1}$.
Then $dx=r^{n-1}d\xi$, where $d\xi$ is the surface-area element
of $S^{n-1}$. Let us denote the left-hand side of (\ref{eq: Stein-Weiss})
by $I$. Then we have 
\begin{align*}
I & ={\displaystyle \int_{S^{n-1}}\,d\xi\int_{0}^{\infty}r^{n-1}e^{-2\pi ar}e^{-2\pi irt\cdot\xi}\,dr}{\displaystyle =\int_{S^{n-1}}\frac{(n-1)!}{\{2\pi(a+2\pi it\cdot\xi)\}^{n}}\,d\xi.}
\end{align*}
Let $g\in\mathrm{SO}(n)$ be a special orthogonal matrix with the
property $g^{t}(1,0,...,0)=|t|^{-1}t$. Set $\xi=g\eta$. Then $d\xi=d\eta$
and 
\[
t\cdot\xi=|t|\,g{}^{t}(1,0,...,0))\cdot g\eta=|t|{}^{t}(1,0,...,0)\cdot\eta=|t|\eta_{1}.
\]
We have
\[
I=\int_{S^{n-1}}\frac{(n-1)!}{\{2\pi(a+2\pi i|t|\eta_{1})\}^{n}}\,d\eta.
\]
The variable $\eta$ and the surface-area element $d\eta$ can be
written in terms of the angles $\theta_{1},\dots,\theta_{n-2}\in[0,\pi]$
and $\theta_{n-1}\in[0,2\pi]$ in the following way:
\begin{align*}
 & \eta_{1}=\cos\theta_{1},\;\eta_{2}=\sin\theta_{1}\cos\theta_{2},\;\eta_{3}=\sin\theta_{1}\sin\theta_{2}\cos\theta_{3},\dots,\\
 & \eta_{j}=\sin\theta_{1}\sin\theta_{2}\cdots\sin\theta_{j-1}\cos\theta_{j},\dots,\\
 & \eta_{n-1}=\sin\theta_{1}\sin\theta_{2}\cdots\sin\theta_{n-2}\cos\theta_{n-1},\\
 & \eta_{n}=\sin\theta_{1}\sin\theta_{2}\cdots\sin\theta_{n-2}\sin\theta_{n-1},\\
 & d\eta=\prod_{j=1}^{n-1}\sin^{n-1-j}\theta_{j}\,d\theta_{j}.
\end{align*}
We have 
\begin{equation}
I=\frac{(n-1)!}{(2\pi)^{n}}\int_{0}^{2\pi}d\theta_{n-1}\left(\prod_{j=2}^{n-2}\int_{0}^{\pi}\sin^{n-1-j}\theta_{j}\,d\theta_{j}\right)\int_{0}^{\pi}\frac{\sin^{n-2}\theta_{1}\,d\theta_{1}}{(a+i|t|\cos\theta_{1})^{n-1}}.
\end{equation}
By using Propositions \ref{prop:highschool} and \ref{prop: integtrig},
we obtain Theorem \ref{thm:Stein-Weiss}.

\section{Positively homogeneous functions}

The result of this section shall be used in the next section. A function
$f(x)$ on $\mathbb{R}_{+}=\left\{ x>0\right\} $ is said to be \textit{positively
homogeneous }of degree $\delta$ if 
\begin{equation}
f(\lambda x)=\lambda^{\delta}f(x)\label{eq: homogeneous}
\end{equation}
holds for any $x>0,\lambda>0$. A typical example is $x^{\delta}$.
It is the only positively homogeneous function of degree $\delta$
up to a constant factor. Indeed, we have $f(x)=x^{\delta}f(1)$. 

By differentiation in $\lambda$, the equation (\ref{eq: homogeneous})
yields $xf'(\lambda x)=\delta\lambda^{\delta-1}f(x)$. Plugging $\lambda=1$,
we obtain 
\begin{equation}
x\frac{d}{dx}f(x)=\delta f(x),\;x>0.\label{eq:euler}
\end{equation}
Conversely, any solution of (\ref{eq:euler}) is of the form $f(x)=Cx^{\delta}$
and is homogeneous of degree $\delta$. Summing up, positive homogeneity
of degree $\delta$ is equivalent to (\ref{eq:euler}).

Assume $\delta<0$. Then $f(x)=Cx^{\delta}$ diverges as $x\to+0$
if $C\ne0$. In other words, the operator $xd/dx-\delta$ has a trivial
kernel and is injective on the space of bounded differentiable functions
on $0<x<L$, where $L$ is an arbitrary positive constant. 

The differential operator $xd/dx$ is called the Euler operator and
plays an important role in the theory of ordinary differential equations
with regular singularities.

\section{Integrating a rational function of  trigonometric functions\label{sec:Integrating-a-rational}}

In this section, we prove Proposition \ref{prop: integtrig}. It follows
from Proposition \ref{prop: integtrig-1} below by analytic continuation.
Indeed, when $a>0$ is fixed, the integral $\int_{0}^{\pi}\sin^{n}x/(a+b\cos x)^{n+1}\,dx$
is a holomorphic function in the simply connected domain $\mathbb{C}\setminus\left\{ b\le-a,a\le b\right\} .$
\begin{prop}
\label{prop: integtrig-1} If $a>b>0$, we have 
\begin{equation}
\int_{0}^{\pi}\frac{\sin^{n}x}{(a+b\cos x)^{n+1}}\,dx=\frac{\gamma_{n}}{(a^{2}-b^{2})^{(n+1)/2}}\;(n=0,1,2,\dots).\label{eq:integtrig-1}
\end{equation}
\end{prop}

Notice that the $n=0$ case is well-known as an example of residue
calculus (e.g. \cite{Ahlfors}), but here we give a completely different
proof. Proposition \ref{prop: integtrig-1} is equivalent to Proposition
\ref{prop: integtrig-1-1} below. 
\begin{prop}
\label{prop: integtrig-1-1} If $a>b>0$, we have 
\begin{equation}
\int_{0}^{\pi}\frac{\sin^{n}x}{(\sqrt{a}+\sqrt{b}\cos x)^{n+1}}\,dx=\frac{\gamma_{n}}{(a-b)^{(n+1)/2}}\;(n=0,1,2,\dots).\label{eq:integtrig-1-1}
\end{equation}
\end{prop}

Now we prove Proposition \ref{prop: integtrig-1-1}. Let $f(a,b)$
be the left hand side of (\ref{eq:integtrig-1-1}). It is positively
homogeneous in the sense that $f(\lambda a,\lambda b)=\lambda^{-(n+1)/2}f(a,b),\lambda>0$.
On the other hand, when $b$ is fixed and $a$ tends to $\infty$,
we have by Lebesgue's dominated convergence theorem 
\begin{align}
\lim_{a\to\infty}a^{(n+1)/2}f(a,b) & =\int_{0}^{\pi}\lim_{a\to\infty}\frac{a^{(n+1)/2}\sin^{n}x}{(\sqrt{a}+\sqrt{b}\cos x)^{n+1}}\,dx\nonumber \\
 & =\int_{0}^{\pi}\sin^{n}x\,dx=\gamma_{n}.\label{eq:limit}
\end{align}
Set $p=a-b,q=a+b$. Differentiation under the integral sign gives
\begin{align*}
2\frac{\partial f}{\partial q} & =\frac{\partial f}{\partial a}+\frac{\partial f}{\partial b}=\frac{-(n+1)}{2\sqrt{ab}}\int_{0}^{\pi}\frac{(\sqrt{b}+\sqrt{a}\cos x)\sin^{n}x}{(\sqrt{a}+\sqrt{b}\cos x)^{n+2}}\,dx\\
 & =\frac{-1}{2\sqrt{ab}}\left[\frac{\sin^{n+1}x}{(\sqrt{a}+\sqrt{b}\cos x)^{n+1}}\right]_{0}^{\pi}=0.
\end{align*}
So we have discovered a hidden symmetry in $f$. It is independent
of $q$ and is a function of $p$ alone. We denote $f(a,b)=f(p)$
by abuse of notation. The positive homogeneity means $f(p)$ is homogeneous
of degree $-(n+1)/2$. Therefore 
\[
f(a,b)=f(p)=\frac{\text{const.}}{p^{(n+1)/2}}=\frac{\text{const.}}{(a-b)^{(n+1)/2}}.
\]
This constant must be $\gamma_{n}$ because of (\ref{eq:limit}).
The proofs of Propositions \ref{prop: integtrig-1-1}, \ref{prop: integtrig-1},
\ref{prop: integtrig} are now complete.

We give an alternative proof of Proposition \ref{prop: integtrig-1}.
It appeared in \cite{Yamane} as a proof of Proposition \ref{prop: integtrig}.
First, Proposition \ref{prop: integtrig-1} holds for $n=0,1$. If
$n=0$, we have $2\int_{0}^{\pi}=\int_{0}^{2\pi}$ and the right hand
side is well-known in complex analysis (e.g. \cite{Ahlfors}). If
$n=1$, the primitive can be found easily. 

The remaining part can be proved by induction. Let $I_{n}$ be the
left hand side of (\ref{eq:integtrig-1}). If $n\ge2$, integration
by parts gives 
\begin{align*}
I_{n} & =\int_{0}^{\pi}\frac{(-\cos x)'\sin^{n-1}x}{(a+b\cos x)^{n+1}}\,dx\\
 & =(n-1)\left\{ \int_{0}^{\pi}\frac{\sin^{n-2}x}{(a+b\cos x)^{n+1}}\,dx-I_{n}\right\} \\
 & \quad\;+(n+1)b\int_{0}^{\pi}\frac{\sin^{n}x\cos x}{(a+b\cos x)^{n+2}}\,dx.
\end{align*}
The integrals are derivatives of $I_{n-2}$ and $I_{n}$ with respect
to $a$ and $b$ up to constant factors. We get 
\begin{equation}
\Bigl(b\frac{\partial}{\partial b}+n\Bigr)I_{n}=\frac{1}{n}\frac{\partial^{2}}{\partial a^{2}}I_{n-2}\quad(n\ge2).\label{eq:pa}
\end{equation}
If $I_{n-2}=\gamma_{n-2}(a^{2}-b^{2})^{-\frac{n-1}{2}}$, then $I_{n}=\gamma_{n}(a^{2}-b^{2})^{-\frac{n+1}{2}}$
satisfies (\ref{eq:pa}). It is the unique solution of (\ref{eq:pa})
because of the injectivity of $b\partial/\partial b+n$. Induction
proceeds.

\section{Generalization of Proposition \ref{prop: integtrig-1}}

Proposition \ref{prop: integtrig-1} can be found in \cite[3.665.1]{GR}
in a more general form. It states
\begin{prop}
If $a>b>0,\mathrm{Re}\,\mu>0$, we have 
\begin{equation}
\int_{0}^{\pi}\frac{\sin^{\mu-1}x\,dx}{(a+b\cos x)^{\mu}}=\frac{2^{\mu-1}}{\sqrt{(a^{2}-b^{2})^{\mu}}}B\left(\frac{\mu}{2},\frac{\mu}{2}\right)
\end{equation}
\end{prop}

This proposition is a consequence of Propositions \ref{prop:highschool},
\ref{prop: integtrig-1} and Carlson's theorem below. 
\begin{thm}
\label{thm:Carlson}(i) If $f(z)$ is analytic and bounded for $\mathrm{Re}\,z\ge0$
and if $f(n)=0$ for $n=0,1,2,\dots$, then $f(z)$ is identically
zero. 

(ii) If $f(z)$ is analytic and bounded for $\mathrm{Re}\,z>0$ and
if $f(n)=0$ for $n=1,2,\dots$, then $f(z)$ is identically zero.

(iii) If $f(z)$ and $g(z)$ are analytic and bounded for $\mathrm{Re}\,z>0$
and if $f(n)=g(n)$ for $n=1,2,\dots$, then $f(z)=g(z)$ holds for
$\mathrm{Re}\,z>0$.
\end{thm}

\begin{proof}
The weaker version (i) can be found in \cite[p.110]{Askey}. As a
matter of fact, the assumption $f(0)=0$ is superflous: the proof
given there does not use $f(0)=0$. Anyway $n=0$ is removed in (ii),
which is proved in the following way. Let $f(z)$ be analytic and
bounded for $\mathrm{Re}\,z>0$ and assume $f(n)=0$ for $n=1,2,\dots$
If (i) is already known, we can apply it to $g(z)=f(z+1)$ and show
that $f(z)$ is identically zero for $\mathrm{Re}\,z\ge1$. It follows
that $g(z)$ is identically zero for $\mathrm{Re}\,z\ge0$ by analytic
continuation. (This argument shows that it is enough to assume $f(z)=0$
for sufficiently large integers.) Now (iii) is a trivial consequence
of (ii).
\end{proof}

\section{Fourier-Ehrenpreis integral formula}

Proposition \ref{prop: integtrig} was used in \cite{Yamane}. Here
we sketch its main result, which has some similarities with (\ref{eq: Stein-Weiss}). 

Students of math, physics and engineering are taught how to solve
ordinary differential equations of the form $y''+ay'+by=0\,(a,b:\text{const.})$.
Any solution is a superposition (a linear combination) of exponential
solutions (if $a^{2}-4b\ne0$). Such a statement holds true for a
very general class of (systems of) linear partial differential equations
with constant coefficients. It is known as the Fundamental Principle
of Ehrenpreis (\cite{ehren}). In the PDE case, roughly speaking,
any solution is obtained from a superposition of exponential solutions,
where a superposition means an integral with respect to a measure.
The original finding was that such a measure just exists. Later developments
concentrated on expressing such a measure in terms of differential
forms in an more or less concrete way (e.g. \cite{passare}). In the
case of harmonic functions, a very explicit formula was obtained in
\cite{Yamane}.

Let $z=(z_{1},\dots,z_{n})$ be the coordinate of $\mathbb{C}^{n}\,(n\ge3)$.
Set $y=(y_{1},\dots,y_{n})$, where $y_{j}$ is the imaginary part
of $z_{j}$. We consider $V=\bigl\{ z\in\mathbb{C}^{n};z^{2}=\sum_{j=1}^{n}z_{j}^{2}=0\bigr\}$
. The exponential function $f_{z}(t)=\exp(-iz\cdot t)$ is harmonic
in $t$. The main result of \cite{Yamane} states that any harmonic
function is a superposition of such harmonic exponentials. Let $B_{n}$
be the open unit ball of $\mathbb{R}^{n}\,(n\ge3)$. Assume that $u(t)$
is harmonic in $B_{n}$ and is continuous up to the boundary $S^{n-1}$.
Let $f=u|_{\partial B_{n}}$ be its Dirichlet boundary value. Then
in $B_{n}$, we have
\begin{equation}
u(t)=\frac{1}{2(2\pi)^{n-1}}\int_{V}\Bigl(1-\frac{n-2}{2|y|}\Bigr)f(y/|y|)e^{-iz\cdot t}e^{-|y|}\Bigl(\frac{dx\wedge dy}{|y|}\Bigr)^{n-1}.\label{eq:harmmain}
\end{equation}
 In particular, $u(t)$ is a superposition of the exponentials $\exp(-iz\cdot t)$
with $z^{2}=\sum_{j=1}^{n}z_{j}^{2}=0$. and $y/|y|\in\mathrm{supp\,}f$.

\end{document}